\documentclass[11pt,a4paper]{amsart}

\usepackage{graphicx}
\usepackage{amsmath, amssymb}
\usepackage{amsfonts}
\usepackage[arrow,matrix,curve,cmtip,ps]{xy}
\usepackage{amsthm}
\usepackage{enumerate}
\usepackage{hyperref}
\usepackage[dvipsnames]{xcolor}
\usepackage{verbatim}
\usepackage{tikz}
\usepackage{todonotes}
\usetikzlibrary{cd}

\newtheorem{theorem}{Theorem}[section]
\newtheorem{lemma}[theorem]{Lemma}
\newtheorem{conjecture}{Conjecture}
\newtheorem{definition}[theorem]{Definition}

\newtheorem{proposition}[theorem]{Proposition}
\newtheorem{corollary}[theorem]{Corollary}

\theoremstyle{definition}

\theoremstyle{remark}
\newtheorem*{remark}{Remark}

\numberwithin{equation}{section}

\newcommand{\bR}{\mathbb{R}}

\newcommand{\SL}{\mathsf{SL}}
\newcommand{\SO}{\mathsf{SO}}

\newcommand{\Is}{\mathrm{Is}}
\newcommand{\affIs}{\mathrm{Is}}
\newcommand{\nt}{\mathrm{nt}}

\newcommand{\flow}{\mathsf{U}_0\Gamma}
\newcommand{\cflow}{\widetilde{\mathsf{U}_0\Gamma}}
\newcommand{\bdry}{\partial_\infty\Gamma}

\newcommand{\GL}{\mathsf{GL}}
\newcommand{\lspan}{\mathrm{span}}

\newcommand{\red}{\textcolor{red}}

\begin{document}

\title{Affine Anosov representations and Proper actions}

\author{Sourav Ghosh}

\address{Department of Mathematics, Heidelberg University, Germany \footnote{Present Address: Department of Mathematics, Ashoka University, India}}

\email{sourav.ghosh.bagui@gmail.com, sourav.ghosh@ashoka.edu.in}

\author{Nicolaus Treib}
\address{Department of Mathematics, Heidelberg University, Germany}
\email{nicolaus.treib@mathi.uni-heidelberg.de}
\thanks{The authors acknowledge support from U.S. National Science Foundation grants DMS 1107452, 1107263, 1107367 "RNMS: Geometric Structures and Representation Varieties" (the GEAR Network), and from the European Research Council under ERC-Consolidator grant 614733. The first author was supported by a MATCH (Mathematics Center Heidelberg) fellowship, DFG SPP 2026 grant, OPEN/16/11405402 grant and Ashoka University annual research grant. The second author was supported by the Klaus Tschira Foundation.}


\date{\today}

\dedicatory{}


\begin{abstract}
We define the notion of affine Anosov representations of word hyperbolic groups into the affine group $\SO^0(n+1,n)\ltimes\bR^{2n+1}$. We then show that a representation $\rho$ of a word hyperbolic group is affine Anosov if and only if its linear part $\mathtt{L}_\rho$ is Anosov in $\mathsf{SO}^0(n+1,n)$ with respect to the stabilizer of a maximal isotropic plane and $\rho(\Gamma)$ acts properly  on $\mathbb{R}^{2n+1}$.
\end{abstract}

\maketitle
\tableofcontents

\pagebreak

\section*{Introduction}

In this article, we relate the two diverse fields of Anosov representations and Margulis spacetimes. 

Anosov representations of a word hyperbolic group $\Gamma$ into a semisimple Lie group $G$ are certain stable representations whose stability conditions are given in terms of the Gromov flow space $\flow$ of the hyperbolic group. Labourie \cite{Labourie} introduced the notion of an Anosov representation of a closed surface group into $\SL(n,\bR)$ in order to study Hitchin representations. The definition was later extended to representations of word hyperbolic groups into general semisimple Lie groups by Guichard--Wienhard \cite{GW2}. This class of representations has been studied intensively since its introduction, partially due to the fact that it serves as a possible generalization of convex cocompactness to higher rank Lie groups. Recently, Kapovich--Leeb--Porti \cite{KLP} gave a purely geometric characterization of an Anosov representation. However, in this article we will stick to the original dynamical definition of an Anosov representation.

It is natural to wonder what happens to the notion of an Anosov representation when one drops the condition of semisimplicity of the Lie group. In particular, what would be the appropriate notion of an Anosov representation of a word hyperbolic group $\Gamma$ into a Lie group $G\ltimes V$, where $G$ is semisimple and $V$ is a vector space on which $G$ act as linear transformations. Initial work trying to answer some of these questions was done by Ghosh (\cite{Ghosh1},\cite{Ghosh2}). Ghosh introduced the notion of an affine Anosov representation of a free non-abelian group $\Gamma$ into $\mathsf{SO}(2,1)\ltimes \bR^3$ and used it to study Margulis spacetimes.

Margulis spacetimes are quotient manifolds of $\bR^3$ under proper and free actions of a non-abelian free group as affine transformations with discrete linear part. The study of these spaces started with Milnor asking whether the assumption of cocompactness could be dropped from the Auslander conjecture, which states that affine crystallographic groups are virtually solvable. Using Tits' alternative, dropping the assumption of cocompactness implies that a non-abelian free group cannot act properly on $\bR^3$ as affine transformations. However, Margulis (\cite{Margulis1},\cite{Margulis2}) constructed examples of such actions, thereby showing that the assumption of cocompactness cannot be dropped. Moreover, he introduced an invariant which is now called the Margulis invariant, and used it to provide a necessary condition for the affine action of a free non-abelian group to be proper. Previous work of Fried--Goldman \cite{FG}  implies that the linear parts of free non-abelian groups acting properly as affine transformations on $\bR^3$ necessarily lie in some conjugate of $\mathsf{SO}(2,1)$ in $\mathsf{GL}(3,\bR)$. Furthermore, Drumm \cite{Drumm2} gave a complete characterization of the linear parts of proper affine actions of a free non-abelian group on $\bR^3$. He showed that any discrete subgroup of $\mathsf{SO}(2,1)$ can appear as the linear part of such actions.

Subsequently, Abels--Margulis--Soifer \cite{AMS} showed the existence of discrete subgroups of $\mathsf{SO}(n+1,n)\ltimes\bR^{2n+1}$ acting properly on $\bR^{2n+1}$ when $n$ is odd. They also showed the non-existence of discrete subgroups of $\mathsf{SO}(n+1,n)\ltimes\bR^{2n+1}$ acting properly on $\bR^{2n+1}$ when $n$ is even. Recent works of Smilga (\cite{Smilga},\cite{Smilga2},\cite{Smilga3}) extend these results and show existence of discrete subgroups of $G\ltimes V$ acting properly on $V$ under certain assumptions on the semisimple Lie group $G$.

The Margulis invariant spectrum of a representation $\rho:\Gamma\rightarrow\mathsf{SO}(2,1)\ltimes \bR^3$ of a free non-abelian group $\Gamma$ is a function $\alpha_\rho:\Gamma\rightarrow\bR$. While introducing Margulis spacetimes, Margulis made an observation that the Margulis invariant spectrum of a Margulis spacetime is either completely positive or completely negative. The converse of this question is still open although it has been completely answered by Goldman--Labourie--Margulis \cite{GLM} in the case when the linear part of $\rho$ contains no parabolic elements: Given such a representation $\rho:\Gamma\rightarrow\mathsf{SO}(2,1)\ltimes \bR^3$, 
Goldman--Labourie--Margulis constructed a H\"older continuous function $f_\rho:\flow\rightarrow\bR$ which is unique up to Liv\v sic cohomology (for a definition please see \cite{Liv}; see also Definition 1.3.1 of \cite{Ghosh4}) and showed that for any $\gamma\in\Gamma$,
\[\int_\gamma f_\rho=\frac{\alpha_\rho(\gamma)}{t(\gamma)}\]
where $t(\gamma)$ is the period of the periodic orbit of $\gamma$ in $\flow$. They used this identity to extend the normalized Margulis invariant to a map from the space of flow invariant probability measures on $\flow$ to $\bR$. They showed that under the assumption of the linear part being hyperbolic, complete positivity or complete negativity of the extended Margulis invariant is both necessary and sufficient.

Keeping these ideas in mind we extend the notion of an affine Anosov representation to affine groups of the form $\mathsf{SO}^0(n+1,n)\ltimes\mathbb{R}^{2n+1}$ and show that this extended notion is also an open condition. The idea behind the extended Margulis invariant plays a central role in our conception of an affine Anosov representation into $\mathsf{SO}^0(n+1,n)\ltimes\mathbb{R}^{2n+1}$. Moreover, we provide a necessary and sufficient condition for the action of a deformation of an Anosov representation to be proper, in terms of affine Anosov representations. We prove that:
\begin{theorem}
A representation of a word hyperbolic group $\Gamma$ into $\mathsf{SO}^0(n+1,n)\ltimes\mathbb{R}^{2n+1}$ is affine Anosov if and only if its linear part is Anosov with respect to the stabilizer of a maximal isotropic plane and it acts properly on $\mathbb{R}^{2n+1}$.
\end{theorem}

The above theorem is an extension of the corresponding result for the Fuchsian case by Goldman--Labourie--Margulis \cite{GLM}. Danciger--Drumm--Goldman--Smilga \cite{DDGS} recently did a nice survey of this theory which is about to appear for the Margulis Festschrift.
\section*{Organization of the paper}
Sections 1 to 3 cover some preliminaries: In Section 1, we show how to consistently choose orientations on maximal isotropic planes in $\bR^{p,q}$. This is be needed to define the neutral section later on (Definition \ref{def:neutral_section}).\\
Section 2 covers basics about the Gromov geodesic flow. In particular, we show that the flow space is connected (Lemma \ref{Lem:Flowspace_connected}).\\
In Section 3, we first recall the general definition of an Anosov representation before specializing to $\SO^0(n+1,n)$ and showing some contraction properties of associated bundles in Corollary \ref{Cor:Anosov_contraction}.\\
Section 4 contains the main definition of an affine Anosov representation (Definition \ref{Def:affine_Anosov}).\\
In Section 5, we show that affine limit maps for affine representations always exist whenever the linear part of the representation is Anosov (Proposition \ref{prop:limitmap}).\\
In Section 6, we recall some known results about Margulis spacetimes. They are then related to affine Anosov representations in Section 7, where we show in Theorem \ref{thm:proper_affine_Anosov} that affine Anosov representations give rise to Margulis spacetimes. A partial converse of this is shown in Theorem \ref{thm:affine_Anosov_proper}: If the linear part of the holonomy representation of a Margulis spacetime is Anosov, then the holonomy representation is necessarily affine Anosov.\\
The appendix contains some technical points that came up along the way: One is the notion of (AMS)-proximality introduced by Abels--Margulis--Soifer and its consequences. The other subsection deals with the existence of H\"older continuous sections in certain bundles over the flow space which are differentiable along flow lines.

\section*{Acknowledgements} 

We would like to thank Prof. Fran\c cois Labourie and Prof. Anna Wienhard for many helpful conversations. We would also like to thank the anonymous referees for their suggestions.

\section{Consistent orientations}   \label{sec:indefinite_orthogonal}

In this section, we state some results about indefinite orthogonal groups and orientations on certain subspaces of $\bR^{p,q}$ which will prove useful later on. All of this is well-known and included only for the reader's convenience.

Let $p > q$, and let $\bR^{p,q}$ denote the vector space $\bR^{p+q}$, equipped with an indefinite symmetric bilinear form $b_{p,q}$ of signature $(p,q)$. In this section, it will be useful to work in a basis such that the form is given by
\[ I_{p,q} = \begin{pmatrix} 1 \\ & \ddots \\ & & 1 \\ & & & -1 \\ & & & & \ddots \\ & & & & & -1 \end{pmatrix}. \]
Furthermore, let $\pi_p$ and $\pi_q$ denote the two projections corresponding to the splitting
\[ \bR^{p+q} = ( \bR^p \times \{0\} ) \oplus ( \{0\} \times \bR^q ). \]
We will need to consider the space
\[ X_{p,q} := \{ V \subset \bR^{p+q} \ | \ \dim(V) = q, \ b_{p,q}|_{V\times V} \ \text{is negative definite} \}; \]
it is a model for the symmetric space associated to $\SO(p,q)$ and can be identified with $\SO(p,q)/\mathsf{S}(\mathsf{O}(p)\times\mathsf{O}(q))$. It is simply connected, which we can in fact see directly by the following argument.
\begin{lemma}
The space $X_{p,q}$ is contractible.
\begin{proof}
We construct a deformation retraction 
\[ X_{p,q} \times [0,1] \xrightarrow{f} X_{p,q}  \]
onto the point $\{0\} \times \bR^q$, where $f(\cdot,0) = \mathrm{Id}$ and $f(\cdot,1)$ is the constant map with image $\{0\} \times \bR^q$. Decompose any vector $v\in\bR^{p+q}$ as $v=v_p + v_q$, where $v_p = \pi_p(v)$, $v_q = \pi_q(v)$, and consider the map
\begin{align*}
    g: \bR^{p+q} \times [0,1] & \to \bR^{p+q} \\
    (v_p + v_q,t) & \mapsto (1-t)v_p + v_q.
\end{align*}
We observe the following:
\begin{itemize}
 \item If $ b_{p,q}(v\mid v)  < 0$, then $b_{p,q}(g(v,t)\mid g(v,t)) <0 \ \forall t$.
 \item For any $V \in X_{p,q}$, the projection $\pi_q$ restricts to an isomorphism $V \xrightarrow{\cong} \{0\} \times \bR^q$: Otherwise $V$ would have to be contained in the subspace $\pi_q^{-1}(\pi_q(V))$ of signature $(p,q')$ with $\dim\pi_q(V)=q'<q$, a contradiction.
\end{itemize}
Therefore, $g$ induces the desired map $f$.
\end{proof}
\end{lemma}
Using this Lemma, we can describe the two connected components of $\SO(p,q)$. By simple connectivity, it is possible to choose an orientation on each subspace $V\in X_{p,q}$ in a continuous way (which is supposed to mean that for any curve $V_t$, there exist $q$ continuous curves $v_t^i$ such that $(v_t^1,\ldots,v_t^q)$ is a positive basis for $V_t$). An element $A\in\SO(p,q)$ can then either preserve or reverse orientations on the elements of $X_{p,q}$, and a short discussion shows that this distinguishes the two components: \\
Any element $A$ can be deformed to one that fixes $\{0\} \times \bR^q$. To do this, choose a path between $A(\{0\} \times \bR^q)$ and $\{0\} \times \bR^q$ (e.g. the one described in the previous Lemma), then choose a corresponding path $A_t$ in $\SO^0(p,q)$ such that $A_0 = A$ and $A_1$ fixes $\{0\} \times \bR^q$. If $A$ preserves orientations on $X_{p,q}$, then $A_1$ is a transformation in $\SO(p)\times\SO(q)$, which is connected, so $A$ lies in $\SO^0(p,q)$. On the other hand, if $A$ reverses orientations on $X_{p,q}$, it cannot lie in the identity component by continuity of these orientations. By the same argument as before, we can deform $A$ to a fixed standard representative of the second connected component of $\mathsf{S}(\mathsf{O}(p)\times\mathsf{O}(q))$.

In this article, our main interest lies with the space 
\[ \Is_q(\bR^{p,q}) = \{ V \subset \bR^{p+q} \ | \ \dim(V)=q, \ b_{p,q}|_{V\times V} \equiv 0 \} \]
of maximal isotropic subspaces of $\bR^{p,q}$, as well as with stabilizers of such isotropic subspaces in $\SO^0(p,q)$. 

A useful remark is that the above choice of orientations for elements of $X_{p,q}$ induces a consistent choice of orientations for $\Is_q(\bR^{p,q})$ as well:\\
Let $\bR^{p+q} = V_+ \oplus (V_+)^\perp  = V_+ \oplus V_-$ be any orthogonal splitting into a positive definite and a negative definite subspace, and let $\pi_\pm$ denote the corresponding projections. As in the previous Lemma, for any $L\in\Is_q(\bR^{p,q})$, the restriction of $\pi_-$ induces an isomorphism $L\xrightarrow{\cong}V_-$. We use this isomorphism and the orientation on $V_-$ to define an orientation on $L$. Since $X_{p,q}$ is connected and the orientations vary continuously, the induced orientation on $L$ does not depend on the choice of $V_-$. Similarly, this choice of orientations on elements of $\Is_q(\bR^{p,q})$ is continuous.\\
The description of the two connected components of $\SO(p,q)$ now applies in the same way to the action on $\Is_q(\bR^{p,q})$: For $A\in\SO(p,q)$, let $V_-' = A(V_-)$ and $\pi_-'$ be the corresponding projection. Then we have $V_+' = (V_-')^\perp = A(V_+)$, so the diagram
\begin{equation*}
\begin{tikzcd}
L \arrow[r, "A"] \arrow[d, "\pi_-"] & A(L) \arrow[d, "\pi_-'"] \\
V_- \arrow[r, "A"] & V_-'
\end{tikzcd}
\end{equation*}
commutes. Both projections preserve orientation by definition, and the map $A:V_- \to V_-'$ preserves orientation iff $A\in\SO^0(p,q)$, therefore the same is true for the restriction $A:L\to A(L)$.\\
We summarize this paragraph in the following two Propositions:
\begin{proposition}
Let $X_{p,q}$ and $\Is_q(\bR^{p,q})$ be defined as above. Then both $X_{p,q}$ and $\Is_q(\bR^{p,q})$ admit consistent choices of orientations for their elements. The orientations can be chosen to be compatible in the following sense: 

For any $L\in\Is_q(\bR^{p,q})$ and $V_-\in X_{p,q}$, the projection 
\[\pi_-: (V_-)^\perp \oplus V_- \to V_-\]
induces an isomorphism $L\xrightarrow{\cong} V_-$ which is orientation-preserving.
\end{proposition}
\begin{proposition}
A transformation $A\in\SO(p,q)$ belongs to the identity component $\SO^0(p,q)$ if and only if it preserves orientations on (elements of) $X_{p,q}$ and $\Is_q(\bR^{p,q})$.
\end{proposition}

\section{Gromov geodesic flow}  \label{sec:geodesic:flow}

Let $\Gamma$ be a word hyperbolic group and let $\bdry$ be its \emph{Gromov boundary}. The natural action of $\Gamma$ on its boundary has the following north-south dynamics:
\begin{proposition}[{\cite[Proposition 4.2 \& 4.3]{BK}}]
Every element $\gamma\in\Gamma$ of infinite order has exactly two fixed points $\gamma_+,\gamma_-$ in $\bdry$. For any open sets $U,V\subset\bdry$ such that $\gamma_+\in U$, $\gamma_-\in V$, we have $\gamma^n(\bdry - V) \subset U$ for some $n>1$.
\end{proposition}

The action of $\Gamma$ on $\bdry$ extends to a diagonal action of $\Gamma$ on
\begin{align*}
\partial_\infty\Gamma^{(2)} := \partial_\infty\Gamma\times\partial_\infty\Gamma\setminus\{(x,x) \mid x\in \partial_\infty\Gamma\}.
\end{align*}
We denote $\partial_\infty\Gamma^{(2)}\times\mathbb{R}$ by $\cflow$ and for all $(x,y)\in\partial_\infty\Gamma^{(2)}$ and $s,t\in\mathbb{R}$ let
\begin{align*}
    \phi_t: \cflow&\rightarrow \cflow\\
    (x,y,s)&\mapsto (x,y,s+t).
\end{align*}
Gromov \cite{Gromov} showed that there exists a proper cocompact action of $\Gamma$ on $\cflow$ which commutes with the flow $\{\phi_t\}_{t\in\mathbb{R}}$ and the restriction of this action on $\partial_\infty\Gamma^{(2)}$ is the diagonal action. Moreover, there exists a metric on $\cflow$ well defined up to H\"older  equivalence such that the $\Gamma$ action is isometric, the flow $\phi_t$ acts by Lipschitz homeomorphisms and every orbit of the flow $\{\phi_t\}_{t\in\mathbb{R}}$ gives a quasi-isometric embedding. More precisely, the visual metric on $\partial_\infty\Gamma$ is well-defined up to H\"older equivalence (\cite[Theorem 2.18]{BK}), inducing the product metric on $\partial_\infty\Gamma^{(2)}\times\mathbb{R}$ up to H\"older equivalence. Gromov showed in \cite[Corollary 8.3H]{Gromov} that there is a metric which is bi-Lipschitz to the product metric and satisfies the properties above.\\
The flow $\phi_t$ on $\cflow$ gives rise to a flow $\phi_t$, the \emph{Gromov geodesic flow}, on the quotient 
\[\flow := \Gamma\backslash\left(\partial_\infty\Gamma^{(2)} \times \mathbb{R}\right)\] 
which we call the \emph{flow space} of $\Gamma$.

More details about this construction can be found in Champetier \cite{Champetier} and Mineyev \cite{Mineyev}. In particular, the flow space has the following properties which will be important to us later:
\begin{proposition}[{\cite[Theorem 60]{Mineyev}}] \
\begin{enumerate}
 \item The flow space $\flow$ is a proper metric space.
 \item To every element $\gamma\in\Gamma$ of infinite order, we associate its \emph{translation length}
  \[ l(\gamma) = \lim_{n\to\infty} \frac{d(\gamma^nx,x)}{n}, \]
  where $x\in\flow$ is any point. Then we have
  \[ l(\gamma) = \inf_{y\in\flow}(d(y,\gamma y)) \]
  and this infimum is realized on the \emph{axis} $\{ (\gamma_-,\gamma_+,t), t\in\bR \}$.
\end{enumerate}
\end{proposition}

We will also need the following result, which follows from the proof of Lemma 1.3 of \cite{GLM}, using Proposition 4.2 and Theorem 4.3 of \cite{BK}. We give the proof here for the reader's convenience.
\begin{lemma}   \label{Lem:Flowspace_connected}
The space $\flow$ is connected.
\begin{proof}
By Proposition 4.2 (1) and (3) in \cite{BK}, every infinite order element $\gamma\in\Gamma$ has exactly two fixed points $\gamma^\pm \in \partial\Gamma$, and the set
\[ \{\gamma^- \mid \gamma\in\Gamma \ \text{of infinite order} \ \} \subset \partial\Gamma \]
is dense. Fix one such element $\gamma$ and consider the set
\[ U = \Gamma \ \backslash \ \{ (\gamma^-,y,t) \mid y\neq \gamma^-, t\in\bR \} \subset \flow. \]
We will show that it is connected. Assume that $W_1,W_2\subset \flow$ are open sets such that $U = (W_1\cap U) \sqcup (W_2 \cap U)$, and that $W_1$ contains a point of $[ \{ (\gamma^-,\gamma^+,t), t\in\bR \} ]$. Denoting by $\widetilde{W}_i$ the lifts to $\cflow$, we see that 
\[\overrightarrow{\gamma^- \gamma^+}:=[\{ (\gamma^-,\gamma^+,t), t\in\bR \}]\] has to be contained in $\widetilde{W}_1$ since it is connected. Now for any $\gamma^-\neq y\in\partial\Gamma$, consider the set $\overrightarrow{\gamma^- y}$. We have
\[ \lim\limits_{n\to\infty} \gamma^n \cdot \overrightarrow{\gamma^- y} = \lim\limits_{n\to\infty} \overrightarrow{\gamma^- (\gamma^n y)} = \overrightarrow{\gamma^- \gamma^+}, \]
so by openness of $\widetilde{W}_1$, the orbit $\Gamma \cdot \overrightarrow{\gamma^- y}$ has to be contained in $\widetilde{W}_1$. Therefore, $U$ is entirely contained in $W_1$.\\
By Proposition 4.2 (2) in \cite{BK}, the orbit $\Gamma\cdot \gamma^-$ is dense in $\partial\Gamma$, so $U$ is a dense connected subset of $\flow$, which is thus connected as well.
\end{proof}
\end{lemma}


\section{Anosov representations}

In this section, we recall the general definition of an Anosov representation and explain how to obtain a modified contraction/expansion property in our setting that we will need later on. The setup used here is very close to the one in \cite{GW1} and \cite{GW2}, which in turn is a generalization of the original definition in \cite{Labourie}. It should be noted that by now, equivalent definitions avoiding the geodesic flow (which is rather involved when considering general word-hyperbolic groups) have been given in \cite{KLP} and \cite{GGKW}. They are less suited for our purposes, however.


Let $G$ be a semisimple Lie group, $\Gamma$ be a word hyperbolic group and $\varrho:\Gamma\rightarrow G$ be an injective homomorphism. Furthermore, let $(P^+,P^-)$ be a pair of opposite parabolic subgroups of $G$ and
\[ \mathcal{X} \subset G/P^+ \times G/P^- \] 
the unique open $G$-orbit.

Next, we need the geodesic flow. We will use the flow space $\flow$ together with the flow $\phi_t$ introduced in the previous section. It induces a flow $\phi_t$ on the trivial bundle $\cflow \times \mathcal{X}$ by acting as the identity on fibers. This flow then descends to a flow $\phi_t$ on the bundle
\[ \mathsf{P}_\varrho = \Gamma \ \backslash \left( \cflow \times \mathcal{X} \right) \]
over $\flow$, where $\Gamma$ acts on $\flow$ as described in the previous section and via $\varrho$ on $\mathcal{X}$. 
The product structure of $\mathcal{X}$ implies that it comes equipped with two distributions $X^+$ and $X^-$, where $(X^+)_{(gP^+,gP^-)} := \mathsf{T}_{gP^+}G/P^+$, and $(X^-)_{(gP^+,gP^-)} := \mathsf{T}_{gP^-}G/P^-$. Since these distributions are $G$-invariant, they are in particular $\Gamma$-invariant and we can interpret them as vector bundles over $\mathsf{P}_\varrho$, which we will also denote by $X^+$ and $X^-$. The flow $\phi_t$ preserves the product structure of $\mathcal{X}$ as well, so it induces a flow on these vector bundles (using the derivative of $\phi_t$ in fiber directions).\\
Now we are ready to state the definition of an Anosov representation.

\begin{definition}  \label{Def:Anosov} 
A representation $\varrho:\Gamma\to G$ is $(P^+,P^-)$-Anosov if the bundle $\mathsf{P}_\varrho$ admits an Anosov section $\sigma$, i.e. a section $\sigma: \flow \to \mathsf{P}_\varrho$ such that
\begin{itemize}
    \item $\sigma$ is parallel (or locally constant) along flow lines of the geodesic flow, with respect to the locally flat structure on $\mathsf{P}_\varrho$
    \item The flow $\phi_t$ is contracting on the bundle $\sigma^*X^+$ and dilating on the bundle $\sigma^*X^-$.
\end{itemize}
\end{definition}
\begin{remark}
\begin{enumerate}[(i)]
\item The contraction/dilation condition in the definition means the following: Pick any continuous norm $(\| \cdot \|_v)_{v\in\flow}$ on the bundles $\sigma^*X^+$ and $\sigma^*X^-$. Then there exist constants $c,C > 0$ such that, for any $w\in\flow$ and $x\in(\sigma^*X^+)_w$, we have
\[ \|\phi_t(x)\|_{\phi_t(w)} < C\exp(-ct)\|x\|_w \]
for all $t>0$, and similarly for any $y\in(\sigma^*X^-)_w$,
\[ \|\phi_{-t}(y)\|_{\phi_{-t}(w)} < C\exp(-ct)\|y\|_w. \]
By compactness of the base, the choice of norm does not matter.
\item It is sometimes easier in terms of notation to lift $\sigma$ to a $\Gamma$-equivariant section of the trivial bundle $\cflow \times \mathcal{X}$. We will write
\[ \widetilde{\sigma}: \cflow \to \mathcal{X} \]
for the $\Gamma$-equivariant map defining this section. It is constant along flow lines.
\end{enumerate}
\end{remark}


We now turn to the case $G=\SO^0(n+1,n), P^+ = \mathrm{Stab}_G(E)$ for some maximal isotropic subspace $E \in \Is_n(\bR^{n+1,n})$. Then $P^+$ is conjugate to its opposite parabolic $P^-$ and the unique open $G$-orbit $\mathcal{X}$ is identified with the space of transverse pairs $(E,F) \in (\Is_n(\bR^{n+1,n}))^2$. Transversality is equivalent to having a direct sum splitting $\bR^{n+1,n} = E \oplus F^\perp$ in this case.    Our goal for the remainder of this section will be to prove a contraction property that is slightly different from the one in Definition \ref{Def:Anosov}.\\
As we saw before, the homogeneous space $\mathcal{X}$ identifies with the space of transverse pairs of maximal isotropics. We will start by giving a more explicit description of the bundles $\sigma^*X^+$ and $\sigma^*X^-$. For any $(V^+,V^-)\in\mathcal{X}$, a chart for $G/P^+ = \Is_n(\bR^{n+1,n})$ containing the point $V^+$ is given by
\begin{equation} \label{Eq:chart_linear} \{ f \in \mathrm{Hom}(V^+,(V^-)^\perp) \ | \ \forall v,w \in V^+: b(v+f(v)\mid w+f(w)) = 0 \}, \end{equation}
where $b$ denotes the symmetric bilinear form of signature $(n+1,n)$. Therefore, the subspace defined by the first distribution, 
\[(X^+)_{(V^+,V^-)} = \mathsf{T}_{V^+} \Is_n(\bR^{n+1,n}),\] is given by
\[ \{ g \in \mathrm{Hom}(V^+,(V^-)^\perp) \ | \ \forall v,w \in V^+: b(v\mid g(w)) + b(g(v)\mid w) = 0 \}. \]
The section $\sigma$ now allows us to convert this pointwise description into a description of the associated bundle
\[ \mathsf{R}_\varrho = \Gamma \ \backslash \left( \cflow \times \bR^{n+1,n} \right). \]
More precisely, $\widetilde{\sigma}$ defines a $\Gamma$-invariant splitting
\begin{equation}    \label{Eq:splitting_ASec_equivar}
\cflow \times \bR^{n+1,n} = \mathcal{V}^+ \oplus \mathcal{L} \oplus \mathcal{V}^-
\end{equation} 
by choosing, for $\widetilde{\sigma}(v) = (V^+,V^-)$,
\[ \mathcal{V}^+_v = V^+, \quad \mathcal{L}_v = (V^+)^\perp \cap (V^-)^\perp, \quad \mathcal{V}^-_v = V^-. \]
Here, orthogonal complements are taken with respect to the bilinear form $b$. The flow action extends to this (trivial) bundle as well by acting trivially on the fiber component. We remark that $b$ is preserved by the flow, which will be useful later on. The flow, the bilinear form and the splitting then descend to give a flow-invariant splitting of $\mathsf{R}_\varrho$, which we will denote by
\[ \mathsf{R}_\varrho = \mathsf{V}^+ \oplus \mathsf{L} \oplus \mathsf{V}^-. \]
The bundle $\sigma^*X^+$ is now identified with the bundle
\[ \mathrm{Hom}_{b-\mathrm{skew}}(\mathsf{V}^+,\mathsf{L}\oplus\mathsf{V}^-) = \mathrm{Hom}(\mathsf{V}^+,\mathsf{L}) \oplus \mathrm{Hom}_{b-\mathrm{skew}}(\mathsf{V}^+,\mathsf{V}^-), \]
where a transformation $g$ is called ${b-\mathrm{skew}}$ iff $b(\cdot\mid g(\cdot))= -b(g(\cdot)\mid \cdot)$. Similarly,
\[ \sigma^*X^- = \mathrm{Hom}(\mathsf{V}^-,\mathsf{L}) \oplus \mathrm{Hom}_{b-\mathrm{skew}}(\mathsf{V^-},\mathsf{V^+}). \]
The flow $\phi_t$ acts on an element $\psi\in\sigma^*X^\pm$ by
\[ (\phi_t\psi)(x) = \phi_t(\psi(\phi_{-t}x)), \]
and the Anosov property tells us that this action is contracting on $\sigma^*X^+$ and dilating on $\sigma^*X^-$. Since this holds true for any choice of norm, let us first pick an auxiliary positive definite quadratic form $e$ on $\mathsf{R}_\varrho$ such that the splitting above is orthogonal and $e$ agrees with $b$ on $\mathsf{L}$ (this is possible since the fibers of $\mathsf{L}$ are spacelike for $b$). The induced operator norms are our norms of choice for $\sigma^*X^+$ and $\sigma^*X^-$.\\
After this somewhat lengthy setup, we are finally ready to conclude. All norms in the following statements are induced by $e$.
\begin{lemma}
Let $p \in \flow$ and $0\neq v \in \mathsf{V}^+_p$ be arbitrary. Then there exists $\psi \in \mathrm{Hom}(\mathsf{V}^+_p,\mathsf{L}_p)$ such that $0 \neq \psi(v) = l \in \mathsf{L}_p$ and $\| \psi \| = \frac {\| l \|} { \| v \| }$. Analogously, for $w \in \mathsf{V}^-_p$, we find $\varphi \in \mathrm{Hom}(\mathsf{V}^-_p,\mathsf{L}_p)$ such that $\| \varphi \| = \frac {\| l \|} { \| w \| }$.
\begin{proof}
Complete $v$ to an $e$-orthogonal basis of $\mathsf{V}^+_p$, map $v$ to $0\neq l \in \mathsf{L}_p$ and map all other basis vectors to $0$.
\end{proof}
\end{lemma}
\begin{corollary}   \label{Cor:Anosov_contraction}
The bundle $\mathsf{V}^+$ is dilated by the flow $\phi_t$. The bundle $\mathsf{V}^-$ is contracted by the flow $\phi_t$.
\begin{proof}
Let $p \in \flow$ and $v\in \mathsf{V}^+_p$ be arbitrary. We saw earlier that
\[ (\sigma^*X^+)_p = \mathrm{Hom}(\mathsf{V}^+_p,\mathsf{L}_p) \oplus \mathrm{Hom}_{b-\mathrm{skew}}(\mathsf{V}^+_p,\mathsf{V}^-_p). \]
Using the previous lemma, we can therefore pick $\psi \in (\sigma^*X^+)_p$ such that $\| \psi \| = \frac {\| l \|}{\| v \|}$ for some $0\neq l \in \mathsf{L}_p$ (by picking it in the first summand). Then we have 
\[ \frac{\| l \|}{\| \phi_t(v) \|} = \frac{\| \phi_t(l) \|}{\| \phi_t(v) \|} \leq \| \phi_t(\psi) \| < C\exp(-ct)\| \psi \| = C\exp(-ct) \frac {\| l \|}{\| v \|}, \]
where we used the fact that $b$ is preserved by the flow and agrees with $e$ on $\mathsf{L}$ to get the first equality.\\
The proof for $\mathsf{V}^-$ follows in the same way.
\end{proof}
\end{corollary}
Note that contraction/dilation is reversed for the bundles $\mathsf{V}^\pm$. This is consistent because the Anosov property gives contraction of $\sigma^*X^+$, which we identified with a subbundle of $\mathrm{Hom}(\mathsf{V}^+,(\mathsf{V}^-)^\perp) = (\mathsf{V}^+)^* \otimes (\mathsf{V}^-)^\perp$.

\section{Affine Anosov representations}

In this section, we define the notion of affine Anosov representations of a word hyperbolic group $\Gamma$ into the semidirect product $\SO^0(n+1,n)\ltimes\bR^{2n+1}$. In the following, we use the form $b_{n+1,n}$ given by the matrix
\[ J := \begin{pmatrix} & & I_n \\ & 1 \\ I_n \end{pmatrix}, \]
where $I_n$ denotes the $n\times n$ identity matrix. In particular, $\bR^n\times\{0\} = \lspan(e_1,\ldots,e_n)$ and $\{0\}\times\bR^n = \lspan(e_{n+2},\ldots,e_{2n+1})$ are transverse maximal isotropic subspaces:
\[ \bR^{2n+1} = (\bR^n\times\{0\}) \oplus (\{0\}\times\bR^n)^\perp, \]
where both summands are elements of
\[ \Is_n\left(\bR^{2n+1}\right) = \left\{ V \subset \bR^{2n+1} \ | \ \dim(V)=n, \ b_{n+1,n}|_{V\times V} \equiv 0 \right\}. \]
We denote the corresponding transverse parabolic subgroups in $G=\SO^0(n+1,n)$ by $P^+$ and $P^-$:
\begin{align*}
    P^+ & = \mathrm{Stab}_G(\bR^n\times\{0\}) \\
    P^- & = \mathrm{Stab}_G(\{0\}\times\bR^n).
\end{align*}
Then $G/(P^+\cap P^-)$ identifies with pairs $V_1,V_2 \in \Is_n(\bR^{n+1,n})$ such that $\bR^{2n+1} = V_1 \oplus (V_2)^\perp$; will also call such pairs \emph{transverse}. The intersection $P^+ \cap P^-$ is the reductive group $\GL^+(n)$:
\begin{lemma}
With the above notation, $P^+ \cap P^-$ is naturally identified with $\GL^+(n)$.
\begin{proof}
Since any element $X$ of $P^+ \cap P^-$ stabilizes both $\bR^n\times\{0\}$ and $\{0\}\times\bR^n$, it has to be of block form
\[ X = \begin{pmatrix} A_1 & B_1 & 0 \\ 0 & C & 0 \\ 0 & B_2 & A_2 \end{pmatrix}, \]
where $A_i$ are $n\times n$ matrices, $B_i$ are $n\times 1$ and $C$ is $1\times 1$. The equation $JXJ = (X^t)^{-1}$ reduces this further to the form
\[ X = \begin{pmatrix} A \\ & C \\ & & (A^t)^{-1} \end{pmatrix}, \]
where $A\in \GL(n)$ and $C=\pm 1$. Now since $X$ preserves orientation on $\bR^{2n+1}$, $C$ has to be $+1$. Moreover, we saw in section \ref{sec:indefinite_orthogonal} that we can consistently choose orientations on all elements of $\Is_n(\bR^{n+1,n})$, and an element $g\in\SO(p,q)$ preserves these orientations iff it lies in $\SO^0(p,q)$. We conclude that $A\in\GL^+(n)$.

\end{proof}
\end{lemma}
In order to define what an Anosov representation into the affine group should be, we will require a class of subgroups corresponding to parabolic subgroups in reductive Lie groups. To that end, let $E=E^{2n+1}$ denote the affine space modeled on $\bR^{n+1,n}$, and $\affIs_n\left(E\right)$ the set of affine isotropic subspaces. By this we mean all affine subspaces whose underlying linear subspace is $n$-dimensional and isotropic.\\
In the linear case, we can interchangeably speak about either maximal isotropic subspaces or $(n+1)$-dimensional subspaces of signature $(n,1,0)$ -- here, the first number denotes degenerate directions and the second number denotes positive directions. Taking orthogonal complements allows to switch between the two sets, and any element $g\in\SO^0(n+1,n)$ fixing a maximal isotropic subspace also fixes its orthogonal complement.\\
However, this is no longer true in the affine case. Since there is no natural basepoint, there is no canonical way of choosing an orthogonal complement of an affine subspace of type $(n,1,0)$. Our construction will make use of these $(n+1)$-dimensional affine subspaces instead of affine maximal isotropic subspaces.

\begin{definition}  \label{def:neutral_section}
Let $F\subset E$ be an affine subspace of type $(n,1,0)$. Then we call the subgroup
\[ P_{\mathrm{aff}} = \mathrm{Stab}_{G\ltimes\bR^{2n+1}} F \]
a \emph{pseudoparabolic}.\\
Two affine subspaces $A_1,A_2$ of type $(n,1,0)$ will be called \emph{transverse} if their underlying vector subspaces $W_1,W_2$ satisfy $\bR^{2n+1} = W_1 \oplus (W_2)^\perp$. Two pseudoparabolics will be called \emph{transverse} if they are stabilizers of transverse affine subspaces.

\end{definition} 
\begin{remark}
Since the group $G\ltimes\bR^{2n+1}$ acts transitively on affine subspaces of type $(n,1,0)$, all pseudoparabolic subgroups are isomorphic and can be identified (albeit not canonically) with $P\ltimes\bR^{n+1}$, where $P<G$ is the stabilizer of some fixed maximal isotropic subspace of $\bR^{2n+1}$, and the group of translations along the orthogonal complement of the maximal isotropic is identified with $\bR^{n+1}$.
\end{remark}
As in the linear case, for transverse pseudoparabolics $P^\pm_\mathrm{aff}$, the quotient $(G\ltimes\bR^{2n+1}) / (P^+_\mathrm{aff} \cap P^-_\mathrm{aff}) =: \mathcal{X}_\mathrm{aff}$ can be identified with the space of transverse pairs of affine subspaces of type $(n,1,0)$, so we can alternatively view it as a subset
\[ \mathcal{X}_\mathrm{aff} \subset \left( (G\ltimes\bR^{2n+1})/P^+_\mathrm{aff} \times (G\ltimes\bR^{2n+1})/P^-_\mathrm{aff} \right). \]
It is the unique open $(G\ltimes\bR^{2n+1})$-orbit in the space of all pairs of affine subspaces of type $(n,1,0)$.\\
There is a natural map which takes two transverse subspaces as above and returns a (spacelike) vector in the linear part of their intersection, chosen to be normalized and according to an orientation convention:
\begin{definition}[Neutral section]
The \emph{neutral section} is the map
\begin{align*} 
\nu: \mathcal{X} & \to \bR^{2n+1} \\
(V_1,V_2) & \mapsto v,
\end{align*}
where $v\in (V_1)^\perp \cap (V_2)^\perp \cap S^1$ is chosen according to the following orientation convention: From section \ref{sec:indefinite_orthogonal}, we know that we can consistently choose orientations on elements of $\Is_n(\bR^{n+1,n})$. Pick any positively oriented bases $(v_1^1,\ldots,v_n^1)$ and $(v_1^2,\ldots,v_n^2)$ of $V_1$ and $V_2$, and choose $v$ such that $(v_1^1,\ldots,v_n^1,v,v_1^2,\ldots,v_n^2)$ is a positive basis of $\bR^{2n+1}$. This does not depend on the choices involved. We also write $\nu$ for the map
\[ \nu: \mathcal{X}_\mathrm{aff} \to \bR^{2n+1} \]
which takes the linear parts of the two affine subspaces and applies the previous definition.
\end{definition}

\begin{remark}
This neutral section is a natural generalization of the one defined in \cite{GLM} in the case $G=\SO^0(2,1)$.
\end{remark}

We now have to adjust the setup of bundles and flows to the affine case. Recall that the flow space of the hyperbolic group $\Gamma$ is defined as
\[\flow = \Gamma \backslash\cflow :=\Gamma \backslash (\partial_\infty\Gamma^{(2)}\times\bR).\]
We will make use of several bundles over the flow space $\flow$. They are defined in terms of a given representation $\rho:\Gamma\to G\ltimes\bR^{2n+1}$. The first one is the affine equivalent of the bundle $\mathsf{P}_\varrho$,
\[ \mathsf{P}_\rho = \Gamma \ \backslash \left( \cflow \times \mathcal{X}_\mathrm{aff} \right), \]
whose fiber is the space of transverse pairs of affine subspaces of type $(n,1,0)$. Next, we need the bundle
\[ \mathsf{R}_\rho = \Gamma \ \backslash\left( \cflow \times \bR^{2n+1} \right), \]
sections of which plays the role of a basepoint in affine space. In both cases, $\Gamma$ acts diagonally via its natural action on $\cflow$ and via the given representation into $G\ltimes\bR^{2n+1}$ on the second factor. Finally, there is the linear version of the latter bundle,
\[ \mathsf{R}_\varrho = \Gamma \ \backslash\left( \cflow \times \bR^{2n+1} \right), \]
where the action on the second factor is given by the linear part $\varrho:=L(\rho):\Gamma\to G$. Since $L(\rho)$ preserves the form $b_{n+1,n}$ on $\bR^{2n+1}$, this bundle comes equipped with a $(n+1,n)$-form.\\
We can lift a section $\tau: \flow \to \mathsf{R}_\rho$ to a $\Gamma$-equivariant section 
\[ \tilde{\tau}: \cflow\to \cflow\times\bR^{2n+1}. \]
Recall that the flow $\phi_t$ acts on this (trivial) bundle as the geodesic flow on the base and as the identity on the fibers,
\[ \phi_t(x,y,t_0,v) = (x,y,t_0+t,v). \]
This allows us to define the covariant derivative in flow direction by
\[ \nabla_\phi \tilde{\tau} (x,y,t_0) = \left. \frac{d}{dt}\right|_{t=0} \phi_{-t}\tilde{\tau} (x,y,t_0+t) \]
(assuming, for now, that this derivative exists). Since it is $\Gamma$-equivariant, it defines a section $\nabla_\phi \tau: \flow \to \mathsf{R}_\varrho$ where $\varrho=L(\rho)$.
Note that, since the lifted bundle was trivial, we might as well consider $\tilde{\tau}$ as a $\Gamma$-equivariant map $\tau': \cflow \to \bR^{2n+1}$ by projecting to the second factor. In this case, the derivative reads as
\[ \nabla_\phi \tau' (x,y,t_0) = \left. \frac{d}{dt}\right|_{t=0} \tau' (x,y,t_0+t). \]
We still prefer to keep track of the base point, however. This will help to avoid confusion in some formulas and calculations later on.\\
In an analogous way, the $G\ltimes\bR^{2n+1}$-equivariant map $\nu:\mathcal{X}_\mathrm{aff} \to \bR^{2n+1}$ extends to a bundle map
\[ \nu: \mathsf{P}_\rho \to \mathsf{R}_\varrho. \]
Finally, as in the linear case, there are two distributions $X^\pm_\mathrm{aff}$ on $\mathcal{X}_\mathrm{aff}$ coming from its product structure, given by 
\[ (X^\pm_\mathrm{aff})_{(gP^+_\mathrm{aff},gP^-_\mathrm{aff})} = \mathsf{T}_{gP^\pm_\mathrm{aff}} \left( G \ltimes \bR^{2n+1} / P^\pm_\mathrm{aff} \right) \]
for any $g\in G\ltimes\bR^{2n+1}$. Observe that these tangent spaces can be identified with the sum of the tangent space to the linear part and transverse translations: Let $(A^+,A^-)$ be a transverse pair of affine $(n,1,0)$ subspaces, and let $(W^+,W^-)$ be their linear parts. Then $A^+ \cap A^-$ is an affine line parallel to the one dimensional subspace $W^+\cap W^-$. Also, $(W^\pm)^\perp$ are $(n,0,0)$ subspaces with $W^\pm=(W^\pm)^\perp \oplus (W^+\cap W^-)$ and
\[\bR^{2n+1}=(W^+)^\perp \oplus (W^+\cap W^-) \oplus (W^-)^\perp.\]
Hence, $(A^+ \cap A^-) \cap [(W^+)^\perp \oplus (W^-)^\perp]$ is an unique point $(v^+,v^-)$ where $v^\pm\in(W^\pm)^\perp$ and we can write $A^\pm=v^\mp+W^\pm$. In fact, the set of transverse pairs of signature $(n,1,0)$ affine spaces is naturally in bijection with the set
\[\{(W^+,W^-,v^+,v^-)\mid v^\pm\in (W^\pm)^\perp\}.\]
From this, it follows that we can identify 
\[ \mathsf{T}_{(A^+,A^-)}\mathcal{X}_\mathrm{aff} = \mathsf{T}_{(W^+,W^-)}\mathcal{X} \oplus (W^+)^\perp \oplus (W^-)^\perp, \]
and the tangent space splits into the two components
\begin{equation} \label{Eq:tangent_affine} \mathsf{T}_{A^\pm} \left(G \ltimes \bR^{2n+1} / P^\pm_\mathrm{aff} \right)= \mathsf{T}_{W^\pm}(G/P^\pm) \oplus (W^\mp)^\perp. \end{equation}
The two distributions $X^\pm_\mathrm{aff}$ are $G\ltimes\bR^{2n+1}$-invariant and we will see them as vector bundles over $\mathsf{P}_\rho$.\\
Moreover, we have the flow $\phi_t$ acting on the bundles $\mathsf{P}_\rho$ and $\mathsf{R}_\rho$ as Gromov geodesic flow on the base and via parallel transport (with respect to the locally flat structure) on the fibers. Using the derivative of the flow on $\mathsf{P}_\rho$ in fiber directions gives an induced flow on the bundles $X^\pm_\mathrm{aff}$.
\begin{definition}  \label{Def:affine_Anosov}
Let $\Gamma$ be a word hyperbolic group and let 
\[\rho: \Gamma \mapsto G\ltimes\bR^{2n+1}\] be an injective homomorphism. Furthermore, let $P^+_{\mathrm{aff}}, P^-_{\mathrm{aff}}$ be two transverse pseudo-parabolic subgroups.  Then $\rho$ is called \emph{affine Anosov} (with respect to $P^\pm_{\mathrm{aff}}$) if and only if
\begin{enumerate}
\item The bundle $\mathsf{P}_\rho$ admits an \emph{affine Anosov section} $\sigma$, i.e.\ a section $\sigma: \flow \to \mathsf{P}_\rho$ such that:
 \begin{itemize}
     \item $\sigma$ is parallel (locally constant) along flow lines of the geodesic flow on $\flow$.
     \item The flow $\phi_t$ is contracting on the bundle $\sigma^*X^+_\mathrm{aff}$ and dilating on the bundle $\sigma^*X^-_\mathrm{aff}$.
 \end{itemize}
 \item There exists a H\"older section $\tau$ of the bundle $\mathsf{R}_\rho$ which is differentiable along flow lines and satisfies
\[ |b( \nabla_\phi\tau\mid \nu\circ \sigma ) | > 0. \]
\end{enumerate}
\end{definition}
Using \eqref{Eq:tangent_affine}, we see that the bundles $\sigma^*X^\pm_\mathrm{aff}$ split in the following way: Letting $L$ denote the map taking linear parts, we get an induced bundle map
\[ \mathsf{P}_\rho \xrightarrow{L} \mathsf{P}_\varrho\]
where $\varrho$ is the Linear part of $\rho$.
Then $L\circ \sigma$ is a section of $\mathsf{P}_\varrho$, and we have the decomposition
\begin{equation}    \label{Eq:affine_bundle_splitting}
    \sigma^*X^\pm_\mathrm{aff} = (L\circ\sigma)^*X^\pm \oplus \mathsf{V}^\mp.
\end{equation}
\begin{remark}
A representation $\rho$ satisfying Condition 1 of Definition \ref{Def:affine_Anosov} is equivalent to $\varrho$, the linear part of $\rho$, being Anosov with respect to $P^\pm$ (for more details please see Proposition \ref{prop:limitmap} and Theorem \ref{thm:affine_Anosov_proper}).
\end{remark}
%

\section{Affine deformations of Anosov representations}

Let $\Gamma$ be a word hyperbolic group and $G$,$P^\pm$ be as above. Also, let $\rho: \Gamma \rightarrow G\ltimes\bR^{2n+1}$ be an injective homomorphism with its linear part $\varrho:=L(\rho)$ Anosov with respect to $P^\pm$. We denote the Anosov section of $\varrho$ by $\sigma_{\varrho}$ and its lift by $\tilde{\sigma}_{\varrho}: \cflow \rightarrow \mathcal{X}$.
\begin{definition}[Neutralised section]
 A \emph{neutralised section} is a H\"older continuous, $\rho$-equivariant map
\[ f: \cflow \to \cflow\times\bR^{2n+1} \]
which is differentiable along flow lines and satisfies 
\[ \nabla_{{\phi}} f(x,y,t)\in\bR\nu\circ\tilde{\sigma}_{\varrho}(x,y,t) \]
for all $(x,y,t)\in \cflow$.
\end{definition}
\begin{proposition}\label{Prop:neutralised_section}
Every injective homomorphism $\rho: \Gamma \rightarrow G\ltimes\bR^{2n+1}$ whose linear part is Anosov with respect to $P^\pm$ admits neutralised sections.
\end{proposition}
\begin{proof}
Using a partition of unity argument, we can construct a H\"older continuous, $\Gamma$-equivariant section

\[s:\cflow \to \cflow\times\bR^{2n+1}\] 
which is differentiable along flow lines (see Section \ref{sec:sections} for details on this).\\
We want to modify the section $s$ in such a way that it varies only in the direction of the neutral section as we follow any flow line in $\cflow$.\\
Recall that we defined the splitting
\[ \cflow \times \bR^{n+1,n} = \mathcal{V}^+ \oplus \mathcal{L} \oplus \mathcal{V}^- \]
in \eqref{Eq:splitting_ASec_equivar}, where $\mathcal{L}$ is a line bundle spanned by the neutral section. Let $\nabla^+_\phi s(x,y,t)\in\mathcal{V}^+_{(x,y,t)}$ and $\nabla^-_\phi s(x,y,t)\in\mathcal{V}^-_{(x,y,t)}$ be such that 
\[\nabla_\phi s (x,y,t)-\nabla^+_\phi s(x,y,t)-\nabla^-_\phi s(x,y,t)\in\mathcal{L}_{(x,y,t)}.\]
We note that $\nabla_\phi s$ and $\nabla^\pm_\phi s$ are all $\varrho(\Gamma)$ equivariant.

Now using Corollary \ref{Cor:Anosov_contraction} and the fact that \[\phi_t(x,y,t_0,v)=(x,y,t_0+t,v)\] we obtain the following inequalities for $t>0$:
\[ \left\|\phi_{-t}\left(\nabla^+_\phi s(x,y,t_0+t)\right)\right\|_{(x,y,t_0)} \leqslant Ce^{-ct}\left\|\nabla^+_\phi s(x,y,t_0+t)\right\|_{(x,y,t_0+t)}\] for some constants $C,c\in \bR$ and
\[ \left\| \phi_t \left( \nabla^-_\phi s(x,y,t_0-t) \right) \right\|_{(x,y,t_0)}\leqslant Ce^{-ct} \left\|\nabla^-_\phi s(x,y,t_0-t)\right\|_{(x,y,t_0-t)}\] for some constants $C,c\in \bR$.
Moreover, using continuity and $\varrho$-equivariance of $\nabla^\pm_\phi s$, compactness of $\flow$ and the fact that $\|\varrho(\gamma) v\|_{\gamma(x,y,t)}=\|v\|_{(x,y,t)}$ we get that $\|\nabla^\pm_\phi s(x,y,t)\|_{(x,y,t)}$ is bounded by some constant $B$. Hence 
\begin{align*}
&\lim_{k\to\infty}\left\|\int_0^k \phi_{\mp t} \left( \nabla^\pm_\phi s(x,y,t_0\pm t) \right) dt\right\|_{(x,y,t_0)}\\ &\leqslant \lim_{k\to\infty}\int_0^k \left\| \phi_{\mp t} \left( \nabla^\pm_\phi s(x,y,t_0\pm t) \right) \right\|_{(x,y,t_0)} dt\\
&\leqslant \lim_{k\to\infty}\int_0^k Ce^{-ct}\|\nabla^\pm_\phi s(x,y,t_0\pm t)\|_{(x,y,t_0\pm t)} dt\\
&\leqslant \lim_{k\to\infty}\int_0^k Ce^{-ct}B dt <\infty.
\end{align*}
Moreover, we observe that
\begin{align*}
&\left\|\int_k^K \phi_{\mp t} \left( \nabla^\pm_\phi s(x,y,t_0\pm t) \right) dt\right\|_{(x,y,t_0)}\\ &\leqslant \int_k^K \left\| \phi_{\mp t} \left( \nabla^\pm_\phi s(x,y,t_0\pm t) \right) \right\|_{(x,y,t_0)} dt\\
&\leqslant \int_k^K Ce^{-ct}\|\nabla^\pm_\phi s(x,y,t_0\pm t)\|_{(x,y,t_0\pm t)} dt\\
&\leqslant \int_k^K Ce^{-ct}B dt = \frac{1}{c}BC\left(e^{-ck}-e^{-cK}\right) .
\end{align*}
Hence, by the positivity and continuity of the exponential function we have for any $\epsilon>0$ constants $T_\epsilon$ such that for all $K>k>T_\epsilon$
\[\frac{1}{c}BC\left(e^{-ck}-e^{-cK}\right)\leqslant\frac{1}{c}BCe^{-ck}\left(1-e^{-c(K-k)}\right)\leqslant\frac{1}{c}BCe^{-cT_\epsilon}<\epsilon.\]
Therefore, the following map is well defined:
\begin{align*}
    f(x,y,t_0):= s(x,y,t_0) &- \int_0^\infty \phi_t \left( \nabla^-_\phi s(x,y,t_0-t) \right) dt\\ &+ \int_0^\infty \phi_{-t} \left( \nabla^+_\phi s(x,y,t_0+t) \right) dt.
\end{align*}
Now we notice that
\begin{align*}
 \nabla_\phi f(x,y,t_0) = \left.\frac{d}{dt}\right|_{t=0}  \phi_{-t}f&(x, y,t_0+t) \\
= \left.\frac{d}{dt}\right|_{t=0} \phi_{-t} s(x,y,t_0+t) & - \left.\frac{d}{dt}\right|_{t=0} \phi_{-t} \int_0^\infty \phi_l \left( \nabla^-_\phi s(x,y,t_0+t - l) \right) dl\\
& + \left.\frac{d}{dt}\right|_{t=0} \phi_{-t} \int_0^\infty \phi_{-l} \left( \nabla^+_\phi s(x,y,t_0+t+l) \right) dl\\
= \left.\frac{d}{dt}\right|_{t=0} \phi_{-t} s(x,y,t_0+t) & - \left.\frac{d}{dt}\right|_{t=0}  \int_{-t}^\infty \phi_{-t}\phi_{t+l} \left( \nabla^-_\phi s(x,y,t_0 - l) \right) dl\\
&+\left.\frac{d}{dt}\right|_{t=0}  \int_{t}^\infty \phi_{-t}\phi_{t-l} \left( \nabla^+_\phi s(x,y,t_0+l) \right) dl\\
 = \nabla_\phi s(x,y,t_0)-\nabla^-_\phi s(x,y&,t_0) -\nabla^+_\phi s(x,y,t_0). 
\end{align*}
Hence $f$ is a neutralised section as $\mathcal{L}_{(x,y,t)}=\bR\nu\circ\tilde{\sigma}_{\varrho}(x,y,t).$
\end{proof}

\begin{proposition}\label{prop:limitmap}
The bundle $\mathsf{P}_\rho$ corresponding to any injective homomorphism $\rho: \Gamma \rightarrow G\ltimes\bR^{2n+1}$, whose linear part $\varrho$ is Anosov with respect to $P^\pm$, admits an affine Anosov section.

\begin{proof}
Let $\sigma_{\varrho}$ be the Anosov section for the Anosov representation $\varrho$ and let $f_\rho$ be a neutralised section corresponding to $\rho$. By taking orthogonal complements, we think of the lift $\tilde{\sigma}_{\varrho}(u), \ u\in\cflow$ as a pair of $(n,1,0)$ vector subspaces $(V_u,W_u)$ of $\bR^{n+1,n}$. Now let us define the following affine section:
\[\tilde{\sigma}_\rho(u):= (f_\rho(u)+ V_u, f_\rho(u)+ W_u).\]
We observe that 
\begin{align*}
\tilde{\sigma}_\rho (\gamma u) &= (f_\rho(\gamma u)+ V_{\gamma u}, f_\rho(\gamma u)+ W_{\gamma u})\\
&= (\rho(\gamma)f_\rho(u)+\varrho(\gamma)V_u, \rho(\gamma)f_\rho(u)+\varrho(\gamma)W_u)\\
&= (\rho(\gamma)(f_\rho(u)+V_u), \rho(\gamma)(f_\rho(u)+W_u))\\
&= \rho(\gamma) \tilde{\sigma}_\rho (u).
\end{align*}
Moreover, as $f_\rho$ is a neutralised section, $f_\rho(\phi_t u)-f_\rho(u)\in\bR\nu\circ\tilde{\sigma}_{\varrho}(u) $. Hence for some constants $c_1,c_2\in\bR$ we have
\begin{align*}
\tilde{\sigma}_\rho (\phi_t u) &= (f_\rho(\phi_t u)+ V_{\phi_t u}, f_\rho(\phi_t u)+ W_{\phi_t u})\\
&= (f_\rho(u)+ c_1\nu(V_u,W_u) + V_u, f_\rho(u)+ c_2\nu(V_u,W_u) + W_u)\\
&= (f_\rho(u)+ V_u, f_\rho(u)+ W_u) = \tilde{\sigma}_\rho (u).
\end{align*}
Hence, the map $\tilde{\sigma}_\rho$ gives rise to a section 
\[\sigma_\rho: \flow \rightarrow \mathsf{P}_\rho\]
which is parallel along the flow lines of the geodesic flow on $\flow$. It follows from the above construction that $L\circ\sigma_\rho=\sigma_\varrho$. Now we need to show that the bundle $\sigma_\rho^*X^+_\mathrm{aff}$ is contracted by the flow and $\sigma_\rho^*X^-_\mathrm{aff}$ is dilated. Recall from \eqref{Eq:affine_bundle_splitting} that we have a splitting $\sigma_\rho^*X^\pm_\mathrm{aff} = (L\circ\sigma_\rho)^*X^\pm \oplus \mathsf{V}^\mp$. 
By assumption $L\circ\sigma_\rho=\sigma_\varrho$ is a (linear) Anosov section, so the bundles $(L\circ\sigma_\rho)^*X^+$ resp. $(L\circ\sigma_\rho)^*X^-$ are contracted resp. dilated by the flow. By Corollary \ref{Cor:Anosov_contraction}, the bundles $V^-$ resp. $V^+$ are also contracted resp. dilated by the flow, so the result follows.
\end{proof}
\end{proposition}

\section{Margulis spacetimes}

In this section we give a brief overview of Margulis spacetimes and present some known results which are related to our main Theorem. 

Margulis spacetimes have a long history starting with the study of affine crystallographic groups. In 1964, Auslander \cite{Aus} failed at an attempt to prove the following statement which was later given the status of a conjecture bearing Auslander's name by Fried--Goldmann \cite{FG}: 
\begin{conjecture}
Affine crystallographic groups are virtually solvable.
\end{conjecture}
This conjecture is still open although it has been answered in affirmative in the case of $\mathbb{R}^3$ by Fried--Goldman and in the case of $\mathbb{R}^n$ for $n<7$ by Abels--Margulis--Soifer. In \cite{Mil}, Milnor asked the further question whether the assumption of cocompactness could be dropped in the Auslander conjecture.   

Margulis answered Milnor's question in the negative by showing the existence of proper affine actions of non-abelian free groups on $\mathbb{R}^3$. The quotient space of such an action is called a Margulis spacetime.

Moreover, Fried--Goldman (\cite{FG}) showed that the linear parts of the non-abelian free groups acting properly on $\mathbb{R}^3$ as affine transformations lie in some conjugate of $\mathsf{SO}(2,1)$. Subsequently, Abels--Margulis--Soifer showed existence of properly discontinuous actions of non-abelian free subgroups of $\mathsf{SO}^0(n+1,n)\ltimes\mathbb{R}^{2n+1}$ on $\mathbb{R}^{2n+1}$. Recently, Smilga (\cite{Smilga}) showed existence of proper actions on $\mathfrak{g}$ of non-abelian free subgroups of $\mathsf{G}\ltimes\mathfrak{g}$ where $\mathsf{G}$ is any semisimple Lie group acting adjointly on its Lie algebra $\mathfrak{g}$. 

In (\cite{Margulis1},\cite{Margulis2}), Margulis introduced a key tool to decide properness. He introduced certain invariants, called Margulis invariants, which behave like length functions to gauge the properness of an action. 
\begin{definition}
Let $\Gamma$ be a word hyperbolic group and let $\rho:\Gamma\rightarrow\mathsf{SO}^0(n+1,n)\ltimes\mathbb{R}^{2n+1}$
be an injective homomorphism such that its linear part $\varrho$ is Anosov with respect to $P^\pm$. Then the Margulis invariant corresponding to any $\gamma\in\Gamma$, is defined as follows:
\[\alpha_\rho(\gamma):=b(\mathsf{u}_\rho(\gamma)\mid\nu\circ\sigma_{\varrho}(\gamma^-,\gamma^+))\]
where $\sigma_{\varrho}$ is the Anosov section corresponding to $\varrho$ and $\rho(\gamma)=(\varrho(\gamma),\mathsf{u}_\rho(\gamma))$.
\end{definition}
In particular, he proved the following result which we restate here using the terminology of Anosov representations:
\begin{lemma}[Opposite-sign lemma]
Let $\Gamma$ be a non-abelian free group and let $\rho:\Gamma\rightarrow\mathsf{SO}^0(2,1)\ltimes\mathbb{R}^3$
be an injective homomorphism such that its linear part $\varrho$ is discrete. Then $\rho(\Gamma)$ acts properly on $\mathbb{R}^3$ only if its Margulis invariant spectrum is either completely positive or completely negative.
\end{lemma} 
In \cite{GLM} Goldman--Labourie--Margulis introduced and proved the appropriate converse direction of the opposite sign lemma using geodesic currents and a generalised Margulis invariant. They also showed that the space of such Margulis spacetimes is an open and fiber wise convex subset of the tangent bundle of the Fricke-Teichm\"uller space. In (\cite{Ghosh1},\cite{Ghosh2}) Ghosh showed that representations giving rise to Margulis spacetimes with Schottky linear part are Anosov. Proposition 3.0.5 of \cite{Ghosh2} and Propositions 7.1 and 8.1 of \cite{GLM} gives us the following result:
\begin{theorem}
Let $\Gamma$ be a non-abelian free group and let $\rho:\Gamma\rightarrow\mathsf{SO}^0(2,1)\ltimes\mathbb{R}^3$
be an injective homomorphism. Then $\rho$ is affine Anosov with respect to $P^\pm_\mathrm{aff}$ if and only if $\varrho$ is Anosov with respect to $P^\pm$ and $\rho(\Gamma)$ acts properly on $\mathbb{R}^3$.
\end{theorem}

In the next section we extend this Theorem and prove it for injective homomorphisms of any word hyperbolic group $\Gamma$ into $\mathsf{SO}^0(n+1,n)\ltimes\mathbb{R}^{2n+1}$.

\section{Margulis spacetimes vs Affine Anosov representations}

In this section we will show that affine Anosov representations always give rise to Margulis spacetimes and certain Margulis spacetimes always come from affine Anosov representations.

\begin{theorem} \label{thm:proper_affine_Anosov}
Let $\Gamma$ be a word hyperbolic group and let $\rho:\Gamma\rightarrow\mathsf{SO}^0(n+1,n)\ltimes\bR^{2n+1}$ be an injective homomorphism which acts properly on $\mathbb{R}^{2n+1}$ with its linear part $\varrho$ Anosov with respect to $P^\pm$. Then $\rho$ is affine Anosov with respect to $P^\pm_{\mathrm{aff}}$.
\end{theorem}
\begin{proof}
As $\Gamma$ acts properly on $\mathbb{R}^{2n+1}$ we get that $\Gamma$ acts properly on $\partial_\infty\Gamma^{(2)}\times\mathbb{R}^{2n+1}$. Hence $\Gamma$ acts properly on $(\cflow\times\mathbb{R}^{2n+1})/\mathbb{R}$ where the action of $\mathbb{R}$ on $\mathbb{R}^{2n+1}$ is trivial. Now using Lemma 5.2 of \cite{GLM} we get that $\mathbb{R}$ acts properly on 
\[\Gamma\backslash(\cflow\times\mathbb{R}^{2n+1})=\mathsf{R}_\rho\] 
where $\Gamma$ acts on $\mathbb{R}^{2n+1}$ through the representation $\rho$. 

Now assume that $\varrho$ is Anosov with respect to $P^\pm$ but $\rho$ is not affine Anosov with respect to $P^\pm_\mathrm{aff}$. By Proposition \ref{prop:limitmap}, $\rho$ admits an affine Anosov section, hence the second part of Definition \ref{Def:affine_Anosov} must fail: There can be no H\"older section $\tau$ of $\mathsf{R}_\rho$ satisfying
\begin{equation}    \label{Eq:positive_section}
    |b( \nabla_\phi\tau\mid \nu\circ \sigma_\varrho )| > 0.
\end{equation}  
We observe that, by Proposition \ref{Prop:neutralised_section}, there exists a section $\tau$ such that its lift $\tilde{\tau}$ is neutralised. Moreover, Lemma 3 of \cite{GL} together with \eqref{Eq:positive_section} implies that there exists a $\phi$-invariant probability measure $\mu_\tau$ on $\flow$ such that
\[\int b( \nabla_\phi\tau\mid \nu\circ\sigma_\varrho) d\mu_\tau = 0. \] 
Indeed, if not then for any such $\tau$ and for all $\phi$-invariant probability measure $\mu$ on $\flow$ we have
\[I_\tau(\mu):=\int b( \nabla_\phi\tau\mid \nu\circ\sigma_\varrho) d\mu \neq 0 .\]
Moreover, if for some $\phi$-invariant probability measures $\mu_1$ and $\mu_2$ on $\flow$ we have $I_\tau(\mu_1)>0$ and $I_\tau(\mu_2)<0$ then $\mu_0:=\frac{I_\tau(\mu_1)\mu_2-I_\tau(\mu_2)\mu_1}{I_\tau(\mu_1)-I_\tau(\mu_2)}$ is a probability measure for which $I_\tau(\mu_0)=0$. Hence, either $I_\tau(\mu)>0$ for all $\phi$-invariant probability measure $\mu$ on $\flow$ or $I_\tau(\mu)<0$ for all $\phi$-invariant probability measure $\mu$ on $\flow$. In the first case $b( \nabla_\phi\tau\mid \nu\circ\sigma_\varrho)$ is Liv\v sic cohomologous to some positive function $f$ and in the second case it is Liv\v sic cohomologous to some negative function $f$. Now it follows from the definition of Liv\v sic cohomology that there exist a function $g$ such that
\[b( \nabla_\phi\tau\mid \nu\circ\sigma_\varrho) - f= \frac{d}{dt} g.\]
We consider $\tau^\prime:=\tau-g\nu\circ\sigma_\varrho$ and observe that it is a H\"older section of $\mathsf{R}_\rho$ which is neutralised and satisfies
\[b( \nabla_\phi\tau^\prime\mid \nu\circ \sigma_\varrho )=f\]
and $|f|>0$ contradicting our hypothesis.

Let $t>0$ and $p\in\flow$. We define
\[\tau_t(p):= \int_0^t b( \nabla_\phi\tau\mid \nu\circ\sigma_\varrho) (\phi_sp) ds.\]
Hence $\int \tau_t d\mu_\tau =0$ by Fubini's theorem.\\
Moreover, as $\flow$ is connected by Lemma \ref{Lem:Flowspace_connected}, we get that for all $t>0$ there exists $p^\tau_t\in\flow$ such that
\[\tau_t(p^\tau_t)=0.\]
Let $q^\tau_t\in\cflow$ be such that $\pi(q^\tau_t)=p^\tau_t$. As $\tau$ admits neutralised lifts, let $\tilde{\tau}:\cflow\rightarrow \cflow\times \bR^{2n+1}$ be a lift of $\tau$ with $\tilde{\tau}$ neutralised and let $\tilde{\tau}_t:\cflow\to\bR$ be a lift of $\tau_t$. We recall that, if $\tilde{\tau}(x,y,s)=(x,y,s,v)$, then 
\[ \phi_t\tilde{\tau}(x,y,s)= \phi_t(x,y,s,v)=(x,y,s+t,v). \] 
As $\tilde{\tau}$ is neutralised, for all $t>0$ we have
\[\tilde{\tau}(\phi_t q^\tau_t) = \phi_t\tilde{\tau}(q^\tau_t) + \tilde{\tau}_t(q^\tau_t)\nu\circ\sigma_\varrho = \phi_t\tilde{\tau}(q^\tau_t).\] 
Now passing to the quotient we get that for all $t>0$, $\tau(\phi_tp^\tau_t)=\phi_t\tau(p^\tau_t)$, and hence \[\tau(\phi_t p^\tau_t)\in\phi_t\tau(\flow)\cap\tau(\flow)\]
for the compact set $\tau(\flow)\subset\mathsf{R}_\rho$.
Therefore, the $\bR$ action on $\Gamma\backslash(\cflow\times\mathbb{R}^{2n+1})=\mathsf{R}_\rho$ is not proper, a contradiction.
\end{proof}

\begin{lemma}   \label{Lem:Margulis_invariant_length}
Let $\alpha_\rho(\gamma)$ be the Margulis invariant and let $t(\gamma)$ be the period of the closed orbit corresponding to $\gamma$ in $\flow$. Then for all $\tau:\flow\rightarrow\mathsf{R}_\rho$ we have
\[\alpha_\rho(\gamma)=t(\gamma)\int b( \nabla_\phi\tau\mid \nu\circ\sigma_\varrho ) d\mu_\gamma\]
where $\mu_\gamma$ denotes the geodesic current on $\flow$ corresponding to $\gamma$.
\end{lemma}
\begin{proof}
Let $p=(\gamma^-,\gamma^+,s)\in\cflow$, let $\tilde{\sigma}$ be the lift of $\sigma$ and let $f_\tau: \cflow \rightarrow \bR^{2n+1}$ be the map corresponding to the lift of $\tau$. We know that 
\[\alpha_\rho(\gamma)= b( \rho(\gamma)f_\tau(p)-f_\tau(p)\mid \nu\circ\tilde{\sigma}_\varrho(p)).\] 
Moreover,
\begin{align*}
    b( \rho(\gamma)f_\tau(p)-f_\tau(p)\mid \nu\circ\tilde{\sigma}_\varrho(p))&=b( f_\tau(\gamma p)-f_\tau(p)\mid \nu\circ\tilde{\sigma}_\varrho(p))\\
    &=\int_0^{t(\gamma)} b( \nabla_\phi f_\tau\mid \nu\circ\tilde{\sigma}_\varrho) (\phi_tp) dt \\
    &=t(\gamma)\int b( \nabla_\phi\tau\mid \nu\circ\sigma_\varrho) d\mu_\gamma 
\end{align*}
and the result follows.
\end{proof}

\begin{theorem} \label{thm:affine_Anosov_proper}
Let $\Gamma$ be a non--elementary word hyperbolic group and let $\rho:\Gamma\rightarrow G\ltimes\bR^{2n+1}$ be an injective homomorphism which is affine Anosov with respect to $P^\pm_{\mathrm{aff}}$. Then $\rho(\Gamma)$ acts properly on $\bR^{2n+1}$ and the linear part $\varrho$ is Anosov with respect to $P^\pm$.
\end{theorem}
\begin{proof} 
Let $\sigma_\rho$ be the affine Anosov section of $\rho$. Its linear part
\[ \sigma_\varrho = L\circ\sigma_\rho \]
is a section of the bundle $\mathsf P_\varrho$ that is parallel along flow lines. Since the bundles $\sigma_\rho^*X^\pm_\mathrm{aff}$ are contracted/dilated by the flow $\phi_t$, so are $\sigma_\varrho^*X^\pm$ (see \eqref{Eq:affine_bundle_splitting}). Thus $\varrho$ is Anosov with respect to $P^\pm$.

Now assume that $\rho(\Gamma)$ does not act properly on $\bR^{2n+1}$. Then, there exists a sequence $\gamma_m\in\Gamma$ with $\rho(\gamma_m)\to\infty$ and a converging sequence $x_m\to x\in \bR^{2n+1}$ with $\rho(\gamma_m) x_m\rightarrow y\in \bR^{2n+1}$.

First of all, we show that we may assume without loss of generality that
\begin{itemize}
    \item $\gamma_m$ has infinite order,
    \item the endpoints $\gamma_m^\pm\in\bdry$ of the axis of $\gamma_m$ have distinct limits $a^\pm$,
    \item $l(\gamma_m) \to \infty$.
\end{itemize}
By \cite[Theorem 1.7]{GW2}, the image $\varrho(\Gamma)$ is (AMS)--proximal (for a short discussion see \autoref{sec:AMS-proximal}). This implies that for any $r>0, \epsilon>0$, there exists a finite set $S\subset \varrho(\Gamma)$ with the following property: For every $m$, there is an element $s(m)\in S$ such that $s(m)\varrho(\gamma_m)$ is $(r,\epsilon)$--proximal. After taking a subsequence, we may assume that $s(m) = s$ is constant. Let $s'\in\Gamma$ be a preimage of $s$ and set $\gamma_m' = s'\gamma_m$. Then $\gamma_m', \varrho(\gamma_m'), \rho(\gamma_m')$ all have infinite order, we still have $\rho(\gamma_m') \to \infty$ and 
\[\rho(\gamma_m')x_m = \rho(s')\rho(\gamma_m)x_m \to \rho(s')y.\] 
Taking a subsequence again, we can assume that $\gamma_m'^\pm$ converge to $a^\pm$. By $(r,\epsilon)$--proximality, the attracting and repelling maximal isotropics of $\varrho(\gamma_m')$ have transverse limits, so we must have $a^+ \neq a^-$. Therefore, $\mathrm{axis}(\gamma_m')$ converges as a subset of $\cflow$ 
\footnote{with respect to the pointed Hausdorff topology on the set of closed subsets of the metric space $\cflow$. A basis of this topology is given by sets of the form 
\[ N_{K,\epsilon}(A) = \{ B \subset \cflow \ \text{closed} \ \mid d_H(B\cap K,A\cap K) < \epsilon \}, \]
where $K\subset\cflow$ is a compact set, $A\subset\cflow$ is a closed set and $d_H$ is the Hausdorff metric on closed subsets of the compact set $K$. }
, which implies $l(\gamma_m') \to \infty$ since $\Gamma$ acts properly on $\cflow$.

We assume from now on that $\gamma_m$ has the properties listed above. Let $\widetilde\sigma_\varrho\colon\cflow\to\mathcal X$ be the $\varrho$--equivariant map corresponding to $\sigma_\varrho$. Since the subsets $\mathrm{axis}(\gamma_m) \subset \cflow$ converge, we can pick a convergent sequence $p_m \in \mathrm{axis}(\gamma_m), \ p_m \to p$. It follows that
\begin{align*}
    \alpha_\rho(\gamma_m) = & b( \rho(\gamma_m)x_m - x_m \mid \nu\circ\widetilde\sigma_\varrho(p_m) )   
    \xrightarrow{m\to\infty} b(  y- x \mid \nu\circ\widetilde\sigma_\varrho(p) ),
\end{align*} 
so in particular $\alpha_\rho(\gamma_m)$ stays bounded. We now show that this is a contradiction, finishing the proof.

Because $\rho$ is affine Anosov, there is a H\"older section $\tau\colon\flow\to\mathsf R_\rho$ satisfying
\[ b( \nabla_\phi \tau \mid \nu_\varrho ) > 0. \]
Since $\flow$ is compact, $b( \nabla_\phi \tau \mid \nu_\varrho )$ is bounded from below by a constant $M>0$, thus 
\[ \int b( \nabla_\phi \tau \mid \nu_\varrho ) \, \mathrm d \mu_{\gamma_m} \geq M, \]
where $\mu_{\gamma_m}$ is the geodesic current corresponding to $\gamma_m$. Therefore, the fact that every orbit of $\phi$ gives a quasi-isometric embedding and Lemma \ref{Lem:Margulis_invariant_length} imply that $\alpha_\rho(\gamma_m) \to \infty$.
\end{proof}

\section{Appendix}

%
%
%
\subsection{(AMS)-Proximality}    \label{sec:AMS-proximal}
Let $G$ be a semisimple Lie group and $(P^+,P^-)$ a pair of opposite parabolic subgroups of $G$. An element $g \in G$ is called \textit{proximal} relative to $G/P^+$ if $g$ has two transverse fixed points $x^\pm\in G/P^\pm$ and the following holds:
\[ \lim_{n\to\infty}\gamma^nx = x^+  \quad \text{for all $x\in G/P^+$ transverse to $x^-$.} \]
Moreover, a subgroup $H<G$ containing a proximal element is also called \textit{proximal}.

We now turn to a quantitative version of proximality. For any $x^-\in G/P^-$, we define
\[ \nt(x^-) := \{ x\in G/P^+ \mid x \ \text{not transverse to} \ x^- \}.\]
Let $d$ be a Riemannian distance on $G/P^+$ and let $x^\pm\in G/P^\pm$. We fix constants $r,\epsilon > 0$ and consider the neighborhoods
\[ N_\epsilon(x^+) = \{ x \in G/P^+ \mid d(x,x^+) < \epsilon \} \]
and
\[ N_\epsilon(\nt(x^-)) = \{ x \in G/P^+ \mid d(x,\nt(x^-)) < \epsilon \}. \]
An element $g\in G$ is called $(r,\epsilon)$-\textit{proximal} relative to $G/P^+$ if it has two transverse fixed points $x^\pm \in G/P^\pm$ satisfying 
\[ d(x^+,\nt(x^-)) \geq r \]
and the following holds: 
\[ g(N_\epsilon(\nt(x^-))^c) \subset N_\epsilon(x^+). \]

A subgroup $H$ of $G$ is called \textit{(AMS)-proximal} relative to $G/P^+$ if there exist constants $r>0$ and $\epsilon_0>0$ such that for any $\epsilon<\epsilon_0$, there exists a finite set $S = S(r,\epsilon) \subset H$ satisfying the following:
For any $g \in H$, there exists $s\in S$ such that $sg$ is $(r,\epsilon)$-\textit{proximal}. 

Finally, a representation $\varrho:\Gamma\rightarrow G$ is called \textit{(AMS)-proximal} if $\mathsf{ker}(\varrho)$ is finite and $\varrho(\Gamma)$ is (AMS)-proximal.

This definition was introduced by Abels--Margulis--Soifer in \cite{AMSproximal}, where they proved a more general version of the following result:
\begin{theorem}[\cite{AMSproximal}, Theorem 4.1]
Let $H<\SL(n)$ be a strongly irreducible subgroup, i.e. all finite index subgroups of $H$ act irreducibly on $\bR^n$. Assume that $H$ contains a proximal element. Then $H$ is (AMS)-proximal relative to $\bR P^{n-1}$.
\end{theorem}
Subsequently, Guichard--Wienhard used the previous theorem to prove (AMS)-proximality for Anosov representations, which we use in this paper:
\begin{theorem}[\cite{GW2}, Theorem 1.7]
Let $\Gamma$ be a finitely generated word hyperbolic group and $\varrho: \Gamma \to G$ Anosov with respect to $P^\pm$. Then $\varrho$ is (AMS)-proximal with respect to $G/P^\pm$.
\end{theorem}
\subsection{\texorpdfstring{Sections over $\flow$}{Sections over the flow space}} \label{sec:sections}
We now explain how to construct sections over the flow space $\flow$ which are differentiable along flow lines. The construction is based on a partition of unity argument, making sure that the bump functions are differentiable along flow lines. Another issue is that the action of $\Gamma$ on $\cflow$ may have fixed points, so we have to be careful defining ``nice'' neighborhoods.\\
Recall from section \ref{sec:geodesic:flow} that $\cflow$ is equipped with a metric which is unique up to H\"older equivalence. This metric is bi-Lipschitz with respect to the product metric of the visual metric on $\bdry$ and the standard metric on $\bR$. For $x = (a,b,t) \in \cflow$ and $\epsilon > 0$, we define
\[U^\epsilon_{x}:=B_\epsilon(a,b)\times(t-\epsilon,t+\epsilon)\subset\bdry^{(2)}\times\bR,\]
where $B_\epsilon$ denotes the $\epsilon$-ball in $\partial\Gamma^2$. 

As $\Gamma$ acts properly on $\cflow$, stabilizers of points in $\cflow$ are finite. It also allows us to find a good set of neighborhoods: Since
\begin{align*}
    f\colon \Gamma \times \cflow & \to \cflow \times \cflow \\
    (\gamma,x) & \mapsto (\gamma x,x)
\end{align*}
is proper, for any compact neighborhood $x \in K$ of a point $x\in\cflow$, 
\[ \pi_1(f^{-1}(K\times K)) = \{ \gamma\in\Gamma \mid \gamma K \cap K \neq \emptyset \} =: \Gamma_K \]
is finite. If $\Gamma_K \setminus \Gamma_x$ is nonempty, for every $\gamma \in (\Gamma_K \setminus \Gamma_x)$, shrink the neighborhood to $K' \subset K$ such that $\gamma K' \cap K' = \emptyset$. After doing this finitely many times, we can assume that $\Gamma_K = \Gamma_x$. Pick $\epsilon>0$ small enough such that $U^\epsilon_x \subset K$. We distinguish two cases, depending on whether $\Gamma_x$ is trivial or not.
\begin{enumerate}
    \item Assume first that $\Gamma_x = \{1\}$. Write $x = (a,b,t)$. Let $\phi: \bdry^{(2)} \to \bR$ be a H\"older continuous bump function which is positive on $B_\epsilon(a,b)$ and zero elsewhere, and $\psi:\bR\to\bR$ a smooth bump function which is positive on $(t-\epsilon,t+\epsilon)$ and zero elsewhere. Then 
    \begin{align*} 
    \theta: \cflow & \to \bR \\
     (c,d,s) & \mapsto \phi(c,d)\psi(s) 
     \end{align*}
    is a bump function at $x=(a,b,t)$ which is positive on $U^\epsilon_x$ and zero elsewhere, and which is smooth along flow lines. Moreover, the derivative along flow lines is again H\"older continuous and smooth along flow lines. Since $\epsilon$ is chosen such that $\Gamma_{U^\epsilon_x} = \{1\}$, it projects to a bump function at $\pi(x)$ on $\flow$ with the same properties.
    \item Assume now that $\Gamma_x \neq \{1\}$. Let 
    \[ V_x := \bigcup\limits_{\gamma\in\Gamma_x} \gamma U^\epsilon_x, \]
    and observe that $\Gamma_{V_x}=\Gamma_x$. 
    Define $\theta$ as before to be a bump function which is positive on $U^\epsilon_x$ and zero elsewhere, and set
    \[ \vartheta := \sum\limits_{\gamma\in\Gamma_x} \theta\circ\gamma. \]
    Then $\vartheta$ is H\"older continuous, smooth along flow lines, positive on $V_x$ and zero elsewhere. Since $\Gamma_{V_x} = \Gamma_x$ and $\vartheta$ is invariant under $\Gamma_x$, it projects to a bump function on $\flow$ with the same properties. Note that $V_x$ gets arbitrarily small as $\epsilon$ gets close to $0$.
\end{enumerate}
We can use these bump functions to construct sections of the affine bundle 
\[ \mathsf{R}_\rho = \Gamma \backslash \left(\cflow \times \bR^{2n+1}\right) . \]
For every point $z\in\flow$, pick a neighborhood $U_z'$ such that $\mathsf{R}_\rho|_{U_z'}$ is trivial. Then, pick a potentially smaller neighborhood $U_z \subset U_z'$ such that the above construction yields a bump function $\alpha_z:\flow\to\bR$ which is positive on $U_z$ and zero elsewhere. By compactness of $\flow$, finitely many such neighborhoods $U_z$ cover $\flow$. Denote them by $U_{z_i}, \ 1\leq i \leq k$. After normalizing, we may assume that $\sum_i \alpha_{z_i} =1$. Letting $s_{z_i}:U_{z_i} \to \mathsf{R}_\rho|_{U_{z_i}}$ denote a constant section (with respect to some local trivialization), observe that the affine combination
\[ s = \sum\limits_i \alpha_{z_i}s_{z_i} \]
is a well-defined section of $\mathsf{R}_\rho$ which is H\"older continuous and smooth along flow lines. Note that arbitrary addition of sections is not well-defined since the bundle is affine, but affine combinations are.

\bibliography{bibliography.bib}
\bibliographystyle{alpha}

\end{document}